\documentclass{article}

\usepackage[utf8]{inputenc}
\usepackage{amsrefs}
\usepackage{amssymb}
\usepackage{mathrsfs}
\usepackage{bbm}
\usepackage{xcolor}
\usepackage{lscape}
\usepackage{subfig}
\usepackage{titlefoot}
\usepackage{float}

\usepackage{amsthm}
\usepackage{enumitem}
\usepackage{amsmath}

\usepackage{ifthen}

\usepackage{hyperref}
\usepackage{todonotes}

\usepackage[export]{adjustbox}

\theoremstyle{plain}
\newtheorem{theorem}{Theorem}[section]
\newtheorem{lemma}[theorem]{Lemma}
\newtheorem{proposition}[theorem]{Proposition}

\newtheorem{corollary}[theorem]{Corollary}

\theoremstyle{definition}
\newtheorem{definition}[theorem]{Definition}

\newtheorem{question}[theorem]{Question}

\theoremstyle{remark}
\newtheorem{remark}[theorem]{Remark}

\newtheorem*{To show}{To show}

\let\phi\varphi
\renewcommand\emptyset\varnothing

\newcommand\Aut{\operatorname{Aut}}
\newcommand\Inn{\operatorname{Inn}}

\newcommand{\Sub}{\operatorname{Sub}}

\newcommand{\Hit}{\operatorname{Hit}}
\newcommand{\Miss}{\operatorname{Miss}}
\newcommand{\conv}{\operatorname{conv}}

\newcommand\bbR{\mathbf{R}}

\newcommand\bbZ{\mathbf{Z}}

\newcommand\bfY{\mathcal{Y}}

\newcommand\calU{\mathcal{U}}

\newcommand\calW{\mathcal{W}}

\newcommand\calY{\mathcal{Y}}

\title{On compact uniformly recurrent subgroups}

 \author{Pierre-Emmanuel Caprace\thanks{Institut de Recherche en Math\'ematiques et Physique, UCLouvain, Belgium} \and 
 Gil Goffer\thanks{University of California at San Diego, USA} \and 
 Waltraud Lederle\thanks{Intitute of Algebra, TU Dresden, Germany} \and Todor Tsankov\thanks{Université Claude Bernard Lyon 1, Institut Camille Jordan, France}}

\date{\today}

\begin{document}

\maketitle

\begin{abstract}
    Let a group $\Gamma$ act on a paracompact, locally compact, Hausdorff space $M$ by homeomorphisms and let $2^M$ denote the set of closed subsets of $M$. We endow $2^M$ with the Chabauty topology, which is compact and admits a natural $\Gamma$-action by homeomorphisms. We show that for every minimal $\Gamma$-invariant closed subset $\calY$ of $2^M$ consisting of compact sets, the union $\bigcup \calY\subset M$ has compact closure. 
    
    As an application, we deduce that every compact uniformly recurrent subgroup of a locally compact group is contained in a compact normal subgroup. This generalizes a result of U\v{s}akov on compact subgroups whose normalizer is compact. 
\end{abstract}

\section{Introduction} 
Let $G$ be a group. Given a locally compact, Hausdorff $G$-space $M$, the set $2^M$ of closed subsets of $M$, endowed with the Chabauty topology, is naturally a compact $G$-space. This dynamical system has been extensively studied in the case where  $M$ itself is compact, and turned out to play a relevant role in the study of the $G$-space $M$ itself, see for example \cite{Gl75,BauerSigmund1975,AG77}, or the more recent papers \cite{Auslander00,AkinAuslanderNagar2017,Nagar22} and references therein.
A particular emphasis has been put on minimal subsystems of $2^M$, which are sometimes referred to as \emph{quasi-factors}.

If $M$ is compact, every element of $2^M$ is obviously compact.
In this note, we allow $M$ to be non-compact; this is already done in \cite{AkinAuslanderNagar2017} with $G=\bbZ$, but in the study of minimal subsystems of $2^M$ the authors restrict themselves to the case where $M$ is compact.
Our goal is to study those minimal subsystems $\mathcal Y$ of $2^M$ that entirely consist of compact subsets of $M$.
An obvious sufficient condition is that  all elements of $\mathcal Y$ are contained in a compact subset of $M$; in other words, the union $\bigcup \mathcal Y$ has compact closure.

Our main results show that this sufficient condition is also necessary, provided the space $M$ satisfies mellow conditions. One of these conditions is that the space $M$ is  \emph{paracompact}, i.e.,  
every open cover of this space admits a locally finite open refinement. 
\emph{Locally finite} means that each point admits a neighbourhood that intersects only finitely many of the sets of the cover.
For example, any $\sigma$-compact space is paracompact; more generally indeed, a locally compact Hausdorff space is paracompact if and only if it is a disjoint union of open $\sigma$-compact subsets, see Proposition~\ref{prop:paracompactness} below.

We can now state the following. 
\begin{theorem}\label{thm:main thm}
Let a group $\Gamma$ act by homeomorphisms on a paracompact, locally compact, Hausdorff space $M$ 
and consider the associated action of $\Gamma$ on the compact space $2^M$ of closed subsets of $M$.
Let $\calY$ be a minimal, closed, $\Gamma$-invariant subset of $2^M$.

Then all elements of $\calY$ are compact if and only if $\bigcup \calY$ has compact closure.
\end{theorem}


We shall also show that the same conclusion holds if the paracompactness of $M$ is replaced by other formally independent conditions (see Section \ref{sec:special cases}),   notably by the hypothesis that some $Y \in \calY$ has finitely many connected components (see Theorem~\ref{prop:main thm finitely many conn component}). However, we do not know whether the conclusion holds in full generality for a locally compact Hausdorff space $M$ without any further restriction.
\begin{question}
    Can the hypothesis of paracompactness be discarded in Theorem~\ref{thm:main thm}?
\end{question}

Our initial motivation to study this question arose from the special case where the space $M$ is a locally compact group $G$. The conjugation action of $G$ on itself yields a continuous $G$-action on $2^G$, which preserves the closed subset $\Sub(G)$ formed by the closed subgroups of $G$. A minimal subsystem of $\Sub(G)$ is called a \emph{uniformly recurrent subgroup}, or \emph{URS}. URS were first considered by Glasner and Weiss \cite{glasner2015uniformly}. Using Theorem~\ref{thm:main thm}, we establish the following:


\begin{corollary}\label{cor:main cor with URS}
Let $G$ be a locally compact group and   $\calY$ be a URS of $G$ consisting of compact subgroups. Then $\bigcup \mathcal{Y}$ is contained in a compact normal subgroup.    
\end{corollary}


This corollary  emerged as an attempt to identify a purely topological analogue of a theorem by Bader--Duchesne--L\'ecureux \cite{bader2016amenable}, asserting that an amenable invariant random subgroup of a locally compact, second countable group $G$ is almost surely contained in an amenable normal subgroup of $G$.

In the special case where the URS $\calY$ is \textit{homogeneous}, i.e., when it consists of a single conjugacy class (equivalently,  every $Q \in \calY$ has a cocompact normalizer in $G$), we recover a result of V.~U\v sakov, see the lemma on the first page of \cite{Usha1970}.

As an application we get a result about confined subgroups.
A closed subgroup $H \leq G$ is called \emph{confined} if the closure of the conjugacy class of $H$ in $\Sub(G)$ does not contain the trivial subgroup. In particular $H$ itself must be non-trivial.  The following consequence of Corollary~\ref{cor:main cor with URS} is  immediate. 

\begin{corollary}
    Let $G$ be a locally compact group whose only compact normal subgroup is the trivial one. Let $K \leq G$ be a compact subgroup. If $K$ is confined, then the Chabauty-closure of the conjugacy class of $K$  contains a non-compact subgroup. 
\end{corollary}


\subsection*{Acknowledgements}

PEC is supported in part by the FWO and the F.R.S.-FNRS under the EOS programme (project ID 40007542).
GG is supported by the University of California President's Postdoctoral Fellowship.
WL has been a F.R.S.-FNRS postdoctoral researcher.

\section{Preliminaries}

Pursuing the aim of making the paper as self-contained as possible, we have included complete proofs of preliminary lemmas whenever we found it justified, without altering the overall conciseness of this note.  

\subsection{Minimal systems}
Let $\Gamma$ be a group and let $X$ be a topological space. When $\Gamma$ acts on $X$ by homeomorphisms we say that $X$ is a \emph{$\Gamma$-space}, and the pair $(X,\Gamma)$ is called a \emph{dynamical system}, or simply a \emph{system}. If $(X,\Gamma)$ is a dynamical system and $Y\subset X$ is a non-empty, closed, $\Gamma$-invariant subset, then $(Y,\Gamma)$ is a \emph{subsystem} of $(X,\Gamma)$.

\begin{definition}
A dynamical system $(X,\Gamma)$ is called \emph{minimal} if $X$ is non-empty and there is no non-empty closed, $\Gamma$-invariant subset $Y \subsetneq X$. Equivalently, every $\Gamma$-orbit is dense in $X$.
\end{definition}




We note that it is a standard application of Zorn's lemma that every 
dynamical system with $X$ compact contains a minimal subsystem.

A subset $S$ of a group $\Gamma$ is called \emph{syndetic} if there is a finite set $F\subset \Gamma$ such that $\Gamma=FS$. 
The following lemma is well-known in many dynamical settings.
\begin{lemma}\label{lem:returning times set is syndetic}
    Let $\Gamma$ be a group and $X$ a compact $\Gamma$-space. If $(X, \Gamma)$ is minimal, then for any point $x\in X$ and any non-empty open subset $U\subset X$, the set $R(x,U):=\{g\in\Gamma \mid gx\in U\}$ is syndetic.
\end{lemma}



\begin{proof}
  Let $U\subset X$ be open. By minimality, $X = \bigcup_{g \in \Gamma} gU$, and by compactness, there is a finite $F \subset \Gamma$ such that $X = \bigcup_{g \in F} gU$. Then, for any $x$, $G = FR(x,U)$.
\end{proof}

Suppose further that $X$ is Hausdorff, then a converse to this lemma also holds. Namely, let $x\in X$ be any point. If for every open neighborhood $U$ of $x$, the set $R(x,U)$ is syndetic, then $(\overline{\Gamma x},\Gamma)$ is minimal.


\subsection{Minimal systems consisting of subsets}
Let $M$ be a locally compact, Hausdorff topological space. We denote by $2^M$ the space of closed subsets of $M$ equipped with the \emph{Chabauty topology}.
This topology is also known as the \emph{Fell topology}, the \emph{topology of closed convergence} or, if expressed in terms of sequences, the \emph{Painlev\'e--Kuratowski convergence}. For an introduction, we refer to~\cite{Beer_Fell} and~\cite{Beer1993}.
A subbasis for the Chabauty topology is given by 
$$\Hit(U):=\{A\in 2^M|A\cap U \neq \emptyset\} \quad \text{and} \quad \Miss(K):=\{A\in 2^M|A\cap K=\emptyset\},$$ 
where $U$ ranges over all open and $K$ over all compact subsets of $M$.
The space $2^M$ is compact, and the group of homeomorphisms of $M$ acts on $2^M$ by homeomorphisms. 

In particular, given a group $\Gamma$, if $M$ is a $\Gamma$-space, then  $2^M$ is also a $\Gamma$-space. By convention, when referring to a system of the form $(2^M,\Gamma)$, with $M$ a $\Gamma$-space, we mean that the $\Gamma$-action on $2^M$ is the one induced from the $\Gamma$-action on $M$.
The goal of  this article is to address the following.

\begin{question}\label{ques:main}
    Let  $(\calY,\Gamma)$ be  a minimal subsystem of $(2^M,\Gamma)$, whose  elements are all compact. When is the closure of $\bigcup \calY$ compact?
\end{question}

In this paper, we only find several  sufficient conditions that are independent of each other. The question whether the statement holds without any  condition on $M$ whatsoever remains open. 

\begin{remark}\label{rmk:emptyset}
We note that Question \ref{ques:main} has a trivial answer if 
 for every compact set $K \subset M$ there is $\gamma \in \Gamma$ with $K \cap \gamma K = \emptyset$. 
 Then $\emptyset \in \overline{\Gamma K}$ for all compact $K \subset M$ and thus
 $(\{\emptyset\},\Gamma)$ is the only minimal subsystem of $(2^M,\Gamma)$ consisting of compact sets. A non-paracompact example of this kind will be discussed in Section~\ref{sec:special cases}.
\end{remark}

The following two lemmas, regarding the closure of $\bigcup \calY$, will be used later.

\begin{lemma}\label{lem:minimality of calY}
Let $\Gamma$ be a group, let $M$ be a locally compact, Hausdorff $\Gamma$-space, and let $(\calY,\Gamma)$ be a minimal subsystem of $(2^M,\Gamma)$. 
Then for every $Y \in \calY$, we have $\overline{\bigcup \Gamma Y} = \overline{\bigcup \calY}$.
\end{lemma}
\begin{proof}
The containment $\subseteq$ is clear since $\Gamma Y\subseteq \mathcal{Y}$ for any $Y\in \mathcal{Y}$. The containment $\supseteq$ follows by observing that  $\calY \cap 2^{\overline{\bigcup \Gamma Y}}$ is  a closed, non-empty, $\Gamma$-invariant subset of $\mathcal{Y}$, while $\mathcal{Y}$ is minimal.
\end{proof}

\begin{lemma}\label{lem:nbhd_W_and_cpct}
Let $\Gamma$ be a group, let $M$ be a locally compact, Hausdorff $\Gamma$-space, and let $(\calY,\Gamma)$ be a minimal subsystem of $(2^M,\Gamma)$ whose elements are compact. Suppose that the closure of $\bigcup \calY$ is non-compact. Then 
for every $Y\in \calY$, every non-empty open subset $\calW\subset \calY$, and every compact subset $K\subset M$, there exists $g\in \Gamma$ such that $gY \in \calW$ and $gY\nsubseteq K$.
\end{lemma}
\begin{proof}
Assume that the lemma does not hold. Thus there exists $Y\in \calY$, a non-empty open subset $\calW\subset \calY$,
and a compact $K\subset M$ such that $gY\subseteq K$ for every $g\in R=R(Y,\calW)=\{g\in \Gamma \mid gY\in \calW\}$. In particular, $\bigcup _{g\in R}gY\subseteq K$. By Lemma~\ref{lem:returning times set is syndetic}, there exists a finite subset $F \subset \Gamma$ with $FR=\Gamma$.
It follows that
\begin{align*}
\bigcup_{g \in \Gamma} gY = \bigcup_{f \in F}\bigcup_{g \in R} fgY \subset \bigcup_{f \in F} fK.
\end{align*}
Since $F$ is finite, the latter set is compact. However, by Lemma \ref{lem:minimality of calY}, we have that $\overline{\bigcup_{g\in \Gamma} g Y} =\overline{\bigcup \calY}$, contradicting the assumption that $\overline{\bigcup \calY}$ is not compact.
 \end{proof}

\subsection{Paracompact spaces}

We recall the following elementary fact about paracompact spaces.
\begin{lemma}\label{lem:paracpct_cover}
    Let $M$ be a paracompact, locally compact, Hausdorff space.
Let $\calU$ be a locally finite cover of $M$, and let $K\subset M$ be a compact subset. Then the set $\{U \in \calU : \overline U \cap K \neq \emptyset\}$ is finite.
\end{lemma}
\begin{proof}
    Suppose $\{U_n\in \calU :  n\in \mathbb{N}\}$ are infinitely many sets from $\calU$ such that $\overline U_n\cap K\neq \emptyset$ for all $n\in \mathbb{N}$. For every $n\in \mathbb{N}$, let $x_n\in \overline U_n\cap K$. Let $x\in K$ be an accumulation point of $\{x_n : n\in \mathbb{N}\}$. Any neighborhood $U$ of $x$ contains $x_n$ for infinitely many indices $n$, and so $U\cap U_n\neq \emptyset$ for infinitely many indices $n$. This contradicts the assumption that $\calU$ is locally finite.
\end{proof}

Lastly, we will use the following characterizations of paracompact spaces.
For the equivalence of (i) and (iii), see \cite[Chapitre 1, §9]{bourbaki2007topologie}.
That (i) implies (ii) is immediate, and (ii) implies (iii) as a consequence of Lemma \ref{lem:paracpct_cover}.

\begin{proposition}\label{prop:paracompactness}
    For a locally compact, Hausdorff topological space $M$, the following are equivalent:
    \begin{enumerate}[label=(\roman*)]
        \item $M$ is paracompact.
        \item $M$ admits a locally finite open cover in which every set has compact closure.
        \item $M$ admits a partition into $\sigma$-compact, open sets.
    \end{enumerate}
\end{proposition}

\section{The main result for paracompact spaces and locally compact groups}

\label{sec:URS}

\paragraph{Proof of main theorem.} We begin by proving Theorem \ref{thm:main thm}.

\begin{proof}[Proof of Theorem \ref{thm:main thm}]
Assume by contradiction that the closure of $\bigcup \calY$ is not compact. 
%
Fix a locally finite open cover $\calU$ of $M$ in which every set has compact closure (see Proposition \ref{prop:paracompactness}).
Let $Y\in \calY$.
We will construct by induction a sequence $(U_n)_{n}$ of sets from the cover $\calU$ and a sequence $(g_n)_{n}$ of elements of $\Gamma$ such that $g_n Y \cap U_i \neq \emptyset$ for all indices $n$ and $i$ with $n\geq i$. 

For the base case, we take $g_0 = 1_\Gamma$ and $U_0$ any element of $\calU$ that intersects $Y$. Suppose  now that $U_0, \ldots, U_n$ and $g_n$ have already been constructed. Let $\calW = \bigcap_{i \leq n} \Hit(U_i) \cap \calY$ and note that $\calW$ is non-emtpy because it contains $g_nY$. By Lemma~\ref{lem:nbhd_W_and_cpct}, there exists $g \in \Gamma$ such that $gY \in \calW$ and $gY \nsubseteq \bigcup_{i \leq n} \overline U_i$. Choose a point $z \in gY$ not contained in $\bigcup_{i \leq n} \overline U_i$. As $\calU$ is a cover, there exists $U \in \calU$ containing $z$. By construction, we have $U \neq U_0, \ldots, U_n$ and $gY \cap U \neq \emptyset$. Now we define $g_{n+1} = g$ and $U_{n+1} = U$.

Let $Y_0$ be any accumulation point of the net $\{g_nY : n \in \mathbb N\}$.
Since $\Hit(\overline{U_i})$ is closed, and contains $g_nY$ for every $n\geq i$, it follows that the intersection $Y_0 \cap \overline U_i$ is non-empty  for all $i$. Since $Y_0$ is compact and since the cover $\mathcal U$ is locally finite, this contradicts Lemma~\ref{lem:paracpct_cover}.
\end{proof}

\paragraph{Uniformly recurrent subgroups.}
We now show how this purely topological theorem yields an application to \emph{uniformly recurrent subgroups}, whose definition can be found in the introduction.

A group element is called \emph{elliptic} if it is contained in a compact subgroup. The following basic lemma is concerned with the structure of compact normal subgroups of a locally compact group. For totally disconnected groups, it was proved by U\v{s}akov \cite{Ushakov1962}.

\begin{proposition}[Theorem 5.5 in \cite{wang1971compactness}]\label{prop:Ushakov}
Let $G$ be a locally compact group and let $B \subset G$ be a subset of elliptic elements, invariant under conjugation, whose closure $\overline B$ is compact. Then the subgroup $\langle B \rangle$ generated by $B$ is  normal, and its closure  $\overline{\langle B \rangle}$ is compact.
\end{proposition}

For a group $G$ denote by $\Inn(G)$ the group of inner automorphisms of $G$,
i.e. all automorphisms given by conjugation by a group element,
and by $\Aut(G)$ the group of all automorphisms of $G$. We note the following immediate corollary of Proposition \ref{prop:Ushakov}.

\begin{corollary}\label{cor:Ushakov_Gamma_version}
Let $G$ be a locally compact group.
Let $\Gamma$ be a group of automorphisms of $G$ with $\Inn(G) \leq \Gamma \leq \Aut(G)$.
Let $B \subset G$ be a subset of elliptic elements, with compact closure, that is invariant under $\Gamma$.
Then the group $\langle B \rangle$ generated by $B$ is a $\Gamma$-invariant subgroup of $G$ whose closure is compact.
\end{corollary}

\begin{remark}
Some assumption on $\Gamma$ is clearly necessary in this corollary, since in general, even in a discrete group, two elliptic elements may generate an infinite subgroup.
\end{remark}

Let $G$ be a locally compact group. As in the case of a general topological space $M$, we denote by $2^G$ the space of closed subsets of $G$, equipped with the Chabauty topology. Let $\Sub(G)\subset 2^G$ denote the subspace of closed subgroups of $G$ with the induced topology. It is often called the \emph{Chabauty space} of $G$. It is a closed subset of $2^G$, therefore also compact and Hausdorff. The conjugation $G$-action turns $\Sub(G)$ into a $G$-space.


In what comes next, we wish to apply Theorem \ref{thm:main thm} with $M=G$ a locally compact group. For that purpose, we need the following basic fact.

\begin{proposition}\label{pro:lc group is paracompact}   
Locally compact topological groups are paracompact.
\end{proposition}
\begin{proof}
Let $G$ be a locally compact group and $K$ a compact identity neighbourhood. Let $P = \bigcup_{n \geq 1} (K \cup K^{-1})^n$ be the subgroup of $G$ generated by $K$. Hence $P$ is an open subgroup, and it is $\sigma$-compact. The group $G$ is partitioned into the left cosets of $P$. This proves that $G$ is paracompact by Proposition~\ref{prop:paracompactness}.
\end{proof}



By this proposition, every locally compact group satisfies the assumptions on the space $M$ in Theorem~\ref{thm:main thm}. We derive the following. 

\begin{theorem}\label{thm:main thm with Gamma}
Let $G$ be a locally compact group.
Consider a group $\Gamma$ with $\Inn(G)\leq \Gamma\leq \Aut(G)$. Let $(\bfY,\Gamma)$ be a minimal subsystem of $(\Sub(G),\Gamma)$, such that all elements of $\calY$ are compact. Then $\bigcup \mathcal{Y}$ is contained in a compact, $\Gamma$-invariant subgroup of $G$.
\end{theorem}

\begin{proof}
By Theorem~\ref{thm:main thm}, $\bigcup\bfY$ has compact closure. Since $\bigcup \bfY$ is $\Gamma$-invariant and consists only of elliptic elements, Corollary \ref{cor:Ushakov_Gamma_version} implies that it must be contained in a compact subgroup that is $\Gamma$-invariant.
\end{proof}

We obtain Corollary \ref{cor:main cor with URS} as an immediate consequence. 

\begin{proof}[Proof of Corollary \ref{cor:main cor with URS}]
Apply Theorem~\ref{thm:main thm with Gamma} with the group $\Gamma = \Inn(G)$.
\end{proof}


\section{Some examples of non-paracompact spaces}\label{sec:special cases}

In this final section, we establish other sufficient conditions, providing further partial answers to  Question~\ref{ques:main}. 

\paragraph{The first uncountable ordinal.}
Here is an example for Remark \ref{rmk:emptyset} that is otherwise not covered by our results. Take $M=\omega_1$ the first uncountable ordinal with the order topology, and $\Gamma$ any group of homeomorphisms of $\omega_1$ containing all compactly supported homeomorphisms.
Recall that an element of $2^{\omega_1}$ is compact if and only if it is contained in a closed interval.
Define for every limit ordinal $\alpha$ the map $g_\alpha\colon \omega_1\to \omega_1$ by
\[
     g_\alpha(x) = \begin{cases}
                     \alpha + x + 1 & x < \omega_0 \\
                     \alpha + x & \omega_0 \leq x  \leq \alpha \\
                     n - 1 & x = \alpha + n, \, 0 < n < \omega_0 \\
                     \beta & x = \alpha + \beta, \, \omega_0 \leq \beta \leq \alpha \\
                     x & \text{else.}
                \end{cases}
     \]
     It is a homeomorphism exchanging the clopen intervals $[0,\alpha]$ and $[\alpha + 1,\alpha \cdot 2]$, and fixing everything else. We have that $g_\alpha\in \Gamma$, since it is supported on the compact subset $[0,\alpha \cdot 2]$.

\paragraph{Finitely many connected components.}
In what follows, we do not impose a condition on $M$, but rather on elements of $\calY$.

\begin{theorem}
\label{prop:main thm finitely many conn component}
Let $\Gamma$ be a group and let $M$ be a locally compact, Hausdorff $\Gamma$-space. Let $(\calY,\Gamma)$ be a minimal subsystem of $(2^M,\Gamma)$, such that all elements of $\calY$ are compact. If $\calY$ contains an element with finitely many connected components, then $\bigcup \calY$ has compact closure.
\end{theorem}
\begin{proof}
  Suppose towards a contradiction that $\overline{\bigcup \calY}$ is not compact.
  Let $Y \in \calY$ be an element with $n$ connected components $Y_1, \ldots, Y_n$ and let $U_1, \ldots, U_n$ be open subsets of $M$ with compact closures such that $Y_i \subset U_i$ and $\overline U_i \cap \overline U_j = \emptyset$ for $i \neq j$. Let 
  $$\calW = \big(\bigcap_i \Hit(U_i)\big) \cap \big(\bigcap_i \Miss(\overline{U_i} \setminus U_i)\big) \cap \mathcal Y,$$ 
  and note that $\calW$ is non-empty because $Y \in \calW$. 
  Let $g \in \Gamma$ with $gY \in \calW$.
  Then, for every $i=1,\dots,n$ there is a connected component $gY_j$ of $gY$ with $gY_j \cap U_i \neq \emptyset$, which implies $gY_j \subset U_i$ since $gY_j$ is connected and does not intersect the boundary of $U_i$. Because $U_1,\dots,U_n$ are disjoint and because there are precisely $n$ connected components we get $gY \subset U_1 \cup \dots \cup U_n \subset \overline{U_1} \cup \dots \cup \overline{U_n}$.
  This is a contradiction to Lemma~\ref{lem:nbhd_W_and_cpct}.
  %
\end{proof}

\paragraph{Linear continua.}
A standard example of a space that is not paracompact is the long line.
In this subsection we prove that our main result still holds for this space, or more generally, for any linear continuum.

Recall that a \emph{linear continuum} is a totally ordered space $M$ with the order topology, such that the order is dense (i.e., for all $x,y \in M$ with $x<y$ there is $z \in M$ with $x < z < y$) and such that $M$ has the least upper bound property (i.e., every subset with an upper bound has a least upper bound). Equivalently, it is a totally ordered space with the order topology that is connected. Examples are the real numbers $\bbR$ with the usual order, the long line, and $[0,1]^2$ with the lexicographical ordering (remove the endpoints, $\{(0,0),(1,1)\}$, to make it non-compact).

It follows from the density of the order that every linear continuum is Hausdorff. Moreover, the proof that all closed intervals in $\bbR$ are compact carries over to linear continua, showing that they are also locally compact.

\begin{proposition}\label{prop:main thm ordered space}
Let $\Gamma$ be a group acting by homeomorphisms on a 
linear continuum $(M,<)$. 
Let $(\calY,\Gamma)$ be a minimal subsystem of $(2^M,\Gamma)$, such that all elements of $\calY$ are compact. Then $\bigcup \calY$ has compact closure.
\end{proposition}

\begin{proof}
We recall the following facts, most of which can be found in \cite[§24]{Munkres2000}. The set $M$ has the greatest lower bound property.
A subset of $M$ is connected if and only if it is an interval. A subset of $M$ is compact if and only if it is closed and contained in a bounded interval. 
It is now easy to see from the least upper and greatest lower bound property that every non-empty compact subset has a maximum and a  minimum. 
Also, it is an easy consequence of the Intermediate Value Theorem that every homeomorphism either preserves or reverses the order.
It follows that after passing to an index two subgroup if necessary, we can assume that $\Gamma$ preserves the order on $M$.

We assume that $\emptyset \notin \calY$, since otherwise $\calY = \{\emptyset\}$ and we would be done already.

For a compact subset $A \subset M$ denote by $\conv(A) := [\min A, \max A]$ its convex hull. Since $\Gamma$ preserves the order on $M$, we have that $\conv(gA)=g(\conv(A))$ for all $g\in \Gamma$. The maps $\min, \max \colon 2^M \to M$ are not continuous, not even when restricted to compact sets. 
However, we have the following.

\emph{Claim 1:} The maps $\min$ and $\max$ are continuous on $2^K$ for every compact $K \subset M$.

\emph{Proof:} Fix a compact $K\subset M$, and assume without loss of generality that $K=[k^-,k^+]$ is a closed interval. Let $a\in [k^-,k^+]$ and $A\in 2^K$. We have that $\min A < a \iff A \cap (-\infty, a) \neq \emptyset$ and $\min A > a \iff A \cap [k^-, a] = \emptyset$, both of which are open conditions. The proof for $\max$ is similar. 
\qed

\emph{Claim 2:} There exists $Z \in \calY$ such that $\conv(Z)$ is inclusion-minimal among all $\conv(Y)$ with $Y\in \calY$.

\emph{Proof:} Let $\preceq$ be the partial order on $\calY$ defined by $Y \preceq Y'$ if $\conv(Y) \subset \conv(Y')$. We need to show that there exists a $\preceq$-minimal element in $\calY$. 
By Zorn's lemma, it is enough to show that every $\preceq$-chain in $\calY$ has a lower bound. Let $\{Y_i : i\in I\}$ be such a chain; note that it is also a net (directed downwards). By compactness of $\calY$, the net $\{Y_i : i\in I\}$ admits a limit point $Y\in \calY$; note that $Y\neq \emptyset$, because $\emptyset \notin \calY$. Fix $i_0 \in I$. From the continuity of $\min$ and $\max$ on $2^{Y_{i_0}}$ it follows that $\conv(Y) \subset \conv(Y_i)$ for all $i$ and we are done. \qed

Suppose towards a contradiction that $\bigcup \calY$ is unbounded in $M$. Let $Z$ be as in the last claim and denote $\conv Z = [z^-, z^+]$. Note that, by the inclusion-minimality of $\conv(Z)$, for any $g \in \Gamma$, the sets $\conv(Z)$ and $\conv(gZ)$ cannot be strictly included in one another, i.e., one of the three possibilities must happen: $\conv(Z) = \conv(gZ)$, or $z^- <  gz^-$ and $z^+ < gz^+$, or $z^- > gz^-$ and $z^+ > gz^+$.

First consider the case where $\min \bigcup \Gamma \cdot Z = z^-$. By Lemma~\ref{lem:nbhd_W_and_cpct}, there exists $g_0 \in \Gamma$ such that $g_0Z \nsubseteq \conv(Z)$ and then we must have that $g_0 z^- > z^-$. Again by Lemma~\ref{lem:nbhd_W_and_cpct}, there exists $g_1 \in \Gamma$ such that $g_1 z^- < g_0 z^-$ and $g_1Z \nsubseteq \conv(Z \cup g_0Z)$. Now $g_1z^- < g_0z^-$ and $g_1z^+ > g_0z^+$, so $\conv(g_0Z) \subsetneq \conv(g_1Z)$, implying that $\conv(g_1^{-1}g_0Z) \subsetneq \conv(Z)$, contradiction. We treat the case where $\max \bigcup \Gamma \cdot Z = z^+$ similarly. 

Finally, suppose that there exist $g_1, g_2 \in \Gamma$ such that $g_1  z^- < z^-$ and $g_2 z^+ > z^+$. Again, by Lemma~\ref{lem:nbhd_W_and_cpct}, there exists $g \in \Gamma$ with $g z^- < g_2 z^-$, $g z^+ > g_1 z^+$, and $gZ \nsubseteq [g_1z^-, g_2z^+]$. The other case being symmetric, we may assume that $gz^+ > g_2z^+$. Then $g_2Z \subsetneq gZ$ and we arrive at a contradiction as before.
\end{proof}

\bibliographystyle{alpha}
\bibliography{references}
\end{document}